\documentclass[11pt]{article}

\usepackage[a4paper,margin=1in]{geometry}
\usepackage{amsmath,amssymb,amsthm,mathtools}

\usepackage{tikz}

\usepackage{enumitem}
\usepackage{hyperref}
\usepackage{xcolor}

\hypersetup{
  colorlinks=true,
  linkcolor=blue!60!black,
  citecolor=blue!60!black,
  urlcolor=blue!60!black
}

\newcommand{\keywords}[1]{%
  \par\smallskip
  \noindent\textbf{Keywords: }#1
}

\newtheorem*{ThmA}{Theorem A}
\newtheorem*{ThmB}{Theorem B}

\newtheorem{theorem}{Theorem}[section]
\newtheorem{proposition}[theorem]{Proposition}
\newtheorem{lemma}[theorem]{Lemma}
\newtheorem{corollary}[theorem]{Corollary}
\theoremstyle{definition}

\newcommand{\Lip}{\text{Lip}}

\newcommand{\R}{\mathbb{R}}

\newcommand{\F}{\mathcal{F}}

\newcommand{\diam}{\operatorname{diam}}

\renewcommand{\Cap}{\mathrm{Cap}}
\newcommand{\Qap}{\mathrm{Qap}}
\newcommand{\wt}{\mathrm{w}}

\title{Dyadic potential theory and de Rham functions}

\author{Nicola Arcozzi}
\date{June 2026}

\begin{document}

\maketitle

\begin{abstract}
    We study de Rham functional equations driven by two increasing fractional
linear transformations.
Our main purpose is to relate the singularity theory of the associated
solutions to dyadic potential theory on the binary tree.  We first prove
an existence and uniqueness theorem for increasing, left-continuous
solutions in the full range of linear fractional data, and identify the
trapping region in parameter space where the solution is continuous.

For a large class of parameters we show that the de Rham solution is the
normalized cumulative capacitary function of a multiplicative dyadic
capacity.  This gives a potential-theoretic model for Möbius de Rham
systems.  We then sharpen Okamura's Hausdorff-dimensional estimates
for the singular measure associated with the solution by replacing
Hausdorff dimension with dyadic Riesz capacities at the upper endpoint of Okamura's theorem. 
\end{abstract}

\medskip


\section{Introduction}\label{SectIntroduction}
In \cite{Okamura2014}, Okamura considered a de Rham system having as data two fractional linear transformations.
For $j=0,1$, let $g_j(x)=\frac{x+j}{2}$, $g_j:[0,1]\to[0,1]$.
Suppose $[0,1]\xrightarrow{F_0,F_1}[0,1]$ are fractional transformations ,
\[F_j(x)=\frac{a_jx+b_j}{c_jx+d_j}=\Phi(A_j;x),\text{ where }A_j=\begin{pmatrix}
    a_j&b_j\\
    c_j&d_j
\end{pmatrix},\]
such that
\begin{enumerate}
    \item[(A1)] $0=F_0(0),\ F_0(1)=F_1(0),\ F_1(1)=1$;
    \item[(A2)] $F_0,F_1$ are strictly increasing.
\end{enumerate}
In this article, a {\it solution} of the de Rham system with increasing data $F_0,F_1$,
\begin{equation}\label{eqDR}
    f(g_j(x))=F_j(f(x)),
\end{equation}
is a strictly increasing, left continuous function $[0,1]\xrightarrow{f}[0,1]$ such that \eqref{eqDR} holds. Discontinuous solutions naturally arise, as we shall see, in potential theory. A unique solution $f$ of the IFS \eqref{eqDR} exists which satisfies $f(j)=j$ for $j=0,1$ (Theorem \ref{theodeRhamExists}). de Rham's seminal article \cite{dR1957} considered the case
\[
F_0(x)=rx,\ F_1(x)=(1-r)x+r,
\]
with $0<r<1$.

In \cite{Okamura2014} several results concerning the regularity of $f$ are proved in terms of Hausdorff measure. Further regularity and dimension questions for related de Rham and
graph-directed systems were studied in \cite{Okamura2016,Okamura2020}. A rather general setting is considered in \cite{Okamura2026}.

Here, those results will be connected to dyadic potential theory in two ways.
\begin{enumerate}
    \item[(1)] It is shown (Theorem \ref{theodrCap}) that for a large region in the space of data, the solution $f$ of \eqref{eqDR} can be interpreted as a cumulative capacitary function. One interesting consequence is that, for these data, the function $f(x)$ has several variational interpretations. 
    \item[(2)] Okamura's Theorem A concerning the regularity of a deRham function is sharpened by measuring sets' size at the upper endpoint in terms of capacity, rather than Hausdorff measure (Theorem \ref{theoMain}). 
\end{enumerate}
Our main motivation was understanding how dyadic capacities differ from classical "continuous" Riesz capacities. It has been known for a long time that the dyadic and the classical capacities of a subset of $[0,1]$ are comparable \cite{BP1992}\cite{ARSW2014}, but we will see here that the derivative of the dyadic cumulative capacitary function, to be defined below, is mutually singular with respect to the derivative of the classical cumulative one.
Next, we recall the main result in \cite{Okamura2014} and present our new findings.

The data $(F_0,F_1)$ satisfying (A1) and (A2) depend on three real parameters:  the four entries of each of the matrices $A_0,A_1$ can be multiplied times a nonzero constant, and condition (A1) imposes three linear relations. We will use several parametrizations for the data, the simplest of which is by means of {\it linear parameters},
\begin{equation}\label{eqLinPar}
    A_0=\begin{pmatrix}
        1&0\\ \gamma-\lambda&\lambda
    \end{pmatrix},\ A_1=\begin{pmatrix}
        q-1&1\\ q-\gamma&\gamma
    \end{pmatrix}, \text{ with }(\lambda,\gamma,q)\in\mathcal{L}=(0,\infty)\times(1,\infty)\times(0,\infty).
\end{equation}
We have $\gamma=F_0(1)^{-1}=F_1(0)^{-1}$.

Okamura further assumed an open condition:
\begin{enumerate}
    \item[(A3)] $F_0,F_1$ are strict contractions with respect to the Euclidean metric on $[0,1]$. 
\end{enumerate}

Let $\mu_f$ be the distributional derivative of $f$, the Borel probability measure which is defined by
\[
f(x)=\mu_f([0,x)).
\]
The problem is estimating the minimal size of a Borel set $K\subseteq[0,1]$ such that $\mu_f(K)=1$. Okamura did it in terms of Hausdorff measures, we will do it in terms of set capacities.

Next, we precisely state Okamura's main theorem and introduce some parameters which also appear in our main result. We denote by $\dim_H$ the Hausdorff dimension of a set.
\begin{ThmA}[Okamura 2014]\label{theoA} Suppose (A1),(A2), (A3) hold. Then, 
    there are explicit constants $0\le \theta_0\le\theta_1\le\log 2$ depending on $F_0,F_1$ such that:
    \begin{enumerate}
        \item[(upper)] there exists a Borel $K_0\subseteq[0,1]$ with $\mu_f(K_0)=1$ and $\dim_H(K_0)\le\theta_1/\log 2$;
        \item[(lower)] if $K\subseteq[0,1]$ is a Borel subset  such that $\dim_H(K)<\theta_0/\log 2$, then $\mu_f(K)=0$.
    \end{enumerate}
\end{ThmA}
As a consequence,
\begin{equation}\label{eqSynth}
\frac{\theta_0}{\log2}\le\inf\{s\ge0:\mu_f(K)=1\text{ for some Borel $K\subseteq[0,1]$ with $\dim_H(K)=s$}\}\le\frac{\theta_1}{\log 2}.
\end{equation}
The information is sharp when $\theta_0=\theta_1$, the upper estimate provides nontrivial information when $\frac{\theta_1}{\log 2}<1$ and the lower one when $\theta_0>0$.

The numbers $\theta_0,\theta_1$ are expressions of $\gamma$ and of other two real dynamical parameters $\alpha,\beta$ attached to the data $(F_0,F_1)$. Let $r_0\in (-\infty,+\infty]$ be the unique fixed point of $\Phi(^tA_0;\cdot)$ ($r_0=+\infty$ if $a_0=d_0$); and let $r_1$ be the fixed point of $\Phi(^tA_1;\cdot)$ other than $-1$. Then,
\[
\alpha=\min\{0,r_0,r_1\},\ \beta=\max\{0,r_0,r_1\}.
\]
For $y\in(-1,\infty]$, consider the probability distribution
\[
p_0(y)=\frac{y+1}{y+\gamma},\ p_1(y)=1-p_0(y),
\]
and denote the entropy of the probability distribution $\{p,1-p\}$ by 
\[
s(p)=-p\log p-(1-p)\log(1-p)\in[0,\log 2].
\]
Then,
\begin{equation}\label{eqTheta}
\theta_0=\min\{s(p_0(y)):y\in[\alpha,\beta]\}\le\theta_1=\max\{s(p_0(y)):y\in[\alpha,\beta]\}.
\end{equation}
Inspection of the proof shows that Theorem A holds if (A3) is replaced by the weaker condition that a solution $f$ exists and
\begin{enumerate}
    \item[(AA3)] ({\it trapping condition}) $\alpha\ge -1$. 
\end{enumerate}
Let $\mathcal{L}_{\mathrm{tr}}$ be set of the parameters in $\mathcal{L}$ such that (AA3) holds. We will show (Proposition \ref{propTrapped} in \S \ref{SSectCS}) that
\begin{equation}\label{eqLinParTr}
    \mathcal{L}_{\mathrm{tr}}=\{(\lambda,\gamma,q):\lambda\ge 1,q\ge\gamma-1\}.
\end{equation}
We also prove a general existence theorem for increasing Möbius de Rham systems of the form \eqref{eqDR}, which appears not to be explicitly available in the literature.
\begin{theorem}\label{theodeRhamExists}
Let $(\lambda,\gamma,q)\in\mathcal{L}$ and let $F_j=\Phi(A_j;\cdot)$. 
Then \eqref{eqDR} has a unique strictly increasing, left continuous
solution \(f\) satisfying \(f(0)=0\) and \(f(1)=1\), and:
        \begin{enumerate}
        \item[(i)] if $(\lambda,\gamma,q)\in\mathcal{L}_{\mathrm{tr}}$, then $f$ is continuous;
        \item[(ii)] if $(\lambda,\gamma,q)\in\mathcal{L}\setminus\mathcal{L}_{\mathrm{tr}}$, then $f$ is continuous at every non-dyadic point, and has jump discontinuities
at the dyadic points, with the usual one-sided interpretation at the
endpoints. 
    \end{enumerate}
\end{theorem}
The main ingredient of the proof is a nesting property of the IFS, as is the case with many statements of the same kind.

A number of de Rham systems of the form \eqref{eqDR} admit a potential theoretic model. Consider the {\it dyadic tree} $T_\ast$ having root-edge $\ast$. Its edges are finite words ${\mathrm a}=(e_j)_{j=0}^n=\ast e_1\dots e_n$ with $e_0=\ast$ and $e_j\in\{0,1\}$ if $j\ge1$. We denote by $E(T_\ast)$ the edge set and we write $|{\mathrm a}|=n$, and for $j\in\{0,1\}$ let $|{\mathrm a}|_j$ be the number of occurrences of $j$ among $e_1,\dots,e_n$.

The boundary $\partial T_*$ is identified with infinite sequences
\[\zeta = (*,e_1,e_2,\dots),\qquad e_n\in\{0,1\}\ (n\ge 1).\]
The tree's boundary can be identified with the interval $[0,1]$ by means of the usual binary coding map $\partial T_\ast\xrightarrow{\Lambda}[0,1]$,
\[
\Lambda((e_j)_{n=0}^\infty)=\sum_{n=1}^\infty\frac{e_n}{2^n},
\]
which fails to be injective precisely at the dyadic rationals, with the
exceptions of \(0\) and \(1\).

The predecessor set of $\zeta\in\partial T_\ast$ is
\[P_*(\zeta)=\{\zeta(n)=*e_1\dots e_n:\ n\ge 0\}\subset E(T_\ast).\]
We endow $\partial T_*$ with the usual Cantor ultra-metric $\delta$: 
\[
\delta(\zeta,\xi)=2^{-\max\{n:\zeta_n=\xi_n\}},
\]
which makes it into a compact, totally disconnected metric space. Let $E(T_\ast)\xrightarrow{\wt}(0,\infty)$ be a weight on the edge set. 

For a fixed weight $\wt$, define the {\it capacity} of a Borel set $E\subseteq \partial T_*$ to be:
\begin{equation}\label{eq:cap-def}
\Cap_*^{\wt}(E)=\inf\left\{\sum_{{\mathrm a}\in E(T_*)} f({\mathrm a})^2\wt({\mathrm a}):\ f\ge 0,\ \sum_{{\mathrm a}\in P_*(\zeta)} f({\mathrm a})\wt({\mathrm a})\ge 1\ \text{for all }\zeta\in E\right\}.
\end{equation}
Tree capacities of this type are closely connected with random walks
and percolation on trees; see, for instance, \cite{L1992}.
In the special case where $\wt^{\lambda,\mu}(\mathrm{a})=\lambda^{|\mathrm{a}|_0}\mu^{|\mathrm{a}|_1}$, $\lambda,\mu>0$, we write $\Cap_*^{\wt}=\Cap_\ast^{\lambda,\mu}$. A computation given below shows that (Proposition \ref{propTotCap})
\begin{equation}\label{eqTotCap}
    C:=\Cap_\ast^{\lambda,\mu}(\partial T_\ast)=\begin{cases}
        \frac{\lambda+\mu-\lambda\mu}{\lambda+\mu}\text{ if }\lambda+\mu-\lambda\mu>0,\\
        0\text{ if }\lambda+\mu-\lambda\mu\le0.
    \end{cases}
\end{equation}
Let
\[
O_{pot}=\{(\lambda,\mu):\lambda>0,\mu>0,\lambda+\mu-\lambda\mu>0\}.
\]
\begin{theorem}\label{theodrCap}
    Suppose $(\lambda,\mu)\in O_{pot}$ and let $f:[0,1]\to[0,1]$ be the normalized, cumulative capacitary function,
    \[
    f(x)=C^{-1}\Cap_\ast^{\lambda,\mu}(\Lambda^{-1}([0,x))).
    \]
    Then, $f$ is a solution of \eqref{eqDR} with
    \begin{equation}\label{eqAzoeruno}
    A_0=\begin{pmatrix}
        1&0\\
        C&\lambda
    \end{pmatrix},\ A_1=\begin{pmatrix}
        \frac{1}{\mu}&\frac{1}{\lambda}\\
        \frac{C}{\mu}&1+\frac{C}{\lambda}
    \end{pmatrix}.
    \end{equation}
\end{theorem}
The parameters $\lambda,\mu$ appear symmetrically in the weight $\wt$, hence in the corresponding potential theory, but not in the matrices $A_0,A_1$. This is due to the fact that the cumulative capacitary function $f$ considers left intervals $[0,x)$. A nonlinear extension will be stated in Theorem \ref{theoCapNl}.

The matrices $A_0,A_1$ in \eqref{eqAzoeruno} make perfect sense  and the associated fractional linear maps satisfy (A1) and (A2) even when $C\le0$. We have then a set of data which depends on two parameters $(\lambda,\mu)\in(0,\infty)\times(0,\infty)$. In the threshold cases $C=0$ we obtain de Rham's original system.

A third "projective" parameter can be added to the picture. For $t>0$, let 
\begin{equation}\label{eqTa}
    T_t=\begin{pmatrix}
    1&0\\
    1-t&t
\end{pmatrix}
\end{equation}
 The maps $\Phi(T_t;\cdot)$ exhaust the increasing fractional linear transformations mapping $[0,1]$ onto itself. The map $t\mapsto T_t$ is a group isomorphism of the multiplicative $(0,\infty)$ onto the set of such maps. If we let $B_j=T_tA_jT_t^{-1}$, we obtain a family of de Rham data $(\Phi(B_0,\cdot),\Phi(B_1,\cdot))$ which
 depends on three {\it potential theoretic parameters},
 \begin{equation}\label{eqPTpar}
     \mathcal{P}=\{(\lambda,\mu,t):\lambda,\mu,t>0\}.
 \end{equation}
 However, we will see that the data having a potential theoretic parametrization do not exhaust those satisfying (A1) and (A2), and the same is true for those which further satisfy the trapping condition (AA3), which (Proposition \ref{propPTpar}) corresponds to
 \begin{equation}\label{eqPTparTrap}
     \mathcal{P}_{\mathrm{tr}}=\{(\lambda,\mu,t)\in\mathcal{P}:\lambda\ge1,(\lambda+\mu)^2\ge\lambda^2\mu\}.
\end{equation}
Since fractional linear transformations are smooth, they do not alter Hausdorff dimensions, nor the positivity or the vanishing of classical Riesz capacities. The projective parameter $t$ is for this reason harmless, and the conclusions of Theorem \ref{theodeRhamExists} and of Okamura's results hold for $B_0,B_1$ if and only if they hold for $A_0,A_1$.

For \(s\ge0\), define the dyadic Riesz \(s\)-capacity of $E\subseteq\partial T_\ast$ by
\[
\operatorname{Cap}_s(E):=\operatorname{Cap}_\ast^{2^s,2^s}(E)
=
\inf\left\{
\sum_{\alpha\in E(T_\ast)} h(\alpha)^2\,2^{s|\alpha|}:
h\ge0,\ 
\sum_{\alpha\in P_\ast(\zeta)}h(\alpha)\,2^{s|\alpha|}\ge1
\ \text{for every }\zeta\in E
\right\}.
\]
By the classical relations between polarity and Hausdorff dimension (see \S \ref{SSectRBc}), \eqref{eqSynth} is equivalent to 
\begin{equation}\label{eqSynthTwo}
\frac{\theta_0}{\log2}
\le
\inf\left\{
s\ge0:
\exists K\subseteq[0,1]\ \text{Borel},\ 
\mu_f(K)=1,\ 
\Cap_s(\Lambda^{-1}(K))=0
\right\}
\le
\frac{\theta_1}{\log2}.
\end{equation}
Neither Theorem A, nor \eqref{eqSynthTwo}, however, say anything about the size of a Borel set at the endpoints of the interval $[\theta_0/\log 2,\theta_1/\log2]$. About the higher endpoint, we have a potential theoretic result which goes in the direction of singularity.
\begin{theorem}\label{theoMain} Suppose (A1), (A2), and (AA3) hold. Then, there exists a Borel set $K_0\subset[0,1]$ such that $\mu_f(K_0)=1$ and $\Cap_{\frac{\theta_1}{\log2}}(K_0)=0$.
\end{theorem}
Theorem \ref{theoMain} and Theorem A will be somehow sharpened in \S \ref{SSectOpt}, after considering the structure of the data space.

When $\frac{\theta_1}{\log2}=1$, the upper estimate in Theorem A does not carry information, and we will see that this happens in an open region in $\mathcal{L}_{\mathrm{tr}}$. Okamura has a less quantitative result which nonetheless shows that $f$ is singular, but for an exceptional curve in the parameter space.

\begin{ThmB}[Okamura 2014]\label{theoB}
Suppose (A1), (A2), and (AA3) hold. In the linear parameters,
the following alternatives hold.
\begin{enumerate}
    \item[(a)] If both
    \[
        \lambda=2
        \qquad\text{and}\qquad
        q=2(\gamma-1)
    \]
    are satisfied, then \(\mu_f\) is absolutely continuous.

    \item[(b)] If either
    \[
        \lambda\neq 2
        \qquad\text{or}\qquad
        q\neq 2(\gamma-1),
    \]
    then there exists a Borel set \(K_1\subseteq[0,1]\) such that
    \[
        \mu_f(K_1)=1,
        \qquad
        \dim_H(K_1)<1.
    \]
    In particular, \(\mu_f\) is singular.
\end{enumerate}
\end{ThmB}
The one-dimensional manifold where $\mu_f$ is not singular is the orbit of $(\lambda,\gamma,q)=(2,2,2)$ under the action of the group $\{T_t\}_{t>0}$. This result, too, has a potential theoretic, equivalent counterpart.
\begin{corollary}\label{theoAnew}
    Suppose we are in the hypothesis of Theorem B (b). Then, there exist a Borel set $K$ in $[0,1]$ and $0\le s<1$ such that $\mu_f(K)=1$ and $\Cap_s(\Lambda^{-1}(K))=0$.
\end{corollary}
Corollary \ref{theoAnew} follows from Theorem B ike \eqref{eqSynthTwo} follows from Theorem A.

Theorem \ref{theoMain} is in the same spirit of \cite{AC2022}, where a different statement in terms of Hausdorff dimension is sharpened by one expressed in terms of capacity. Actually, it would be possible to prove the potential theoretic \eqref{eqSynthTwo} and Corollary \ref{theoAnew} first, then deducing Theorems A and B from them and from the relations between Hausdorff dimension and polarity of sets.

\vskip0.5cm

The proofs of the theorems will be given in Section \ref{SectProof}, together with other results which might have an independent interest. Several of the main lemmas of \cite{Okamura2014} will be used. In Section \ref{SectFR} we will better analyze the state of the art for different ranges of the data $(F_0,F_1)$, which also leads to an effortless improvement of the thesis of Theorems A and \ref{theoMain}. We also indicate some directions for further inquiry.

In this article we have studied the de Rham systems from \cite{Okamura2014} in some generality.
Several extensions and generalizations could be considered. Dyadic capacities and Riesz capacities can be defined in a vast class of metric spaces, for linear as well as nonlinear potential theories \cite{ARSW2014}. Dyadic capacities satisfy recursive relations which, under some self-similarity assumptions on the involved weights, lead to a multiplicity of de Rham systems. We will return on some of these in a subsequent article. 

\vskip0.5cm

\noindent{\bf Methodological note.}
A generative AI assistant was used during the preparation of this
manuscript for routine algebraic manipulations, numerical and symbolic
checks, bibliographic and terminological queries, and for drafting
preliminary versions of some arguments from proof sketches provided by
the author. All mathematical statements, calculations, proofs, and
bibliographic claims appearing in the paper were subsequently checked,
revised, and are the sole responsibility of the author.

\section{Proofs of the main results}\label{SectProof}

\subsection{The proof of Theorem \ref{theodeRhamExists}}
We give a proof which can be easily modified to obtain similar statements for more general IFS. 
The point of Lemma~\ref{lemmaTransitionRevised} below is that the relevant
contraction is not attached to the individual maps, as in the classical
IFS theorem of Hutchinson~\cite{Hutchinson1981}, but to the occurrence
of transitions in the symbolic itinerary.  This also differs from the weak contraction setting \cite{Hata1985} and from the
general weakly hyperbolic setting of Edalat~\cite{Edalat1996}, and is
adapted to the monotone projective systems considered here.

Let $\partial_{\mathrm{nd}} T_\ast$ be the set of those
$\zeta=(e_n)_{n=0}^\infty$ in $\partial T_\ast$ whose entries are not
eventually constant. Then the restriction of $\Lambda$ to
$\partial_{\mathrm{nd}}T_\ast$ is a homeomorphism onto the set of
non-dyadic points of $(0,1)$. When dealing with non-dyadic points, then, we can work on the symbolic space $\partial_{\mathrm{nd}} T_\ast$ instead of $[0,1]_{\mathrm{nd}}$.

\begin{lemma}[Transition lemma]\label{lemmaTransitionRevised}
Let \((X,d)\) be a complete metric space. Let \(F_0,F_1:X\to X\) be
non-expansive, that is
\[
        d(F_j(x),F_j(y))\le d(x,y),
        \qquad j=0,1.
\]
Assume that there exists \(0<c<1\) such that
\[
        \Lip(F_0\circ F_1)\le c,
        \text{ or }
        \Lip(F_1\circ F_0)\le c.
\]
Assume moreover that one of the images has finite diameter:
\[
        D:=\min\bigl\{
        \diam (F_0\circ F_1)(X),
        \diam (F_1\circ F_0)(X)
        \bigr\}<\infty.
\]
Then, for every \(\omega\in\partial_{\mathrm{nd}}T_\ast\), the limit
\[
        \pi(\omega)
        =
        \lim_{n\to\infty}
        F_{\omega_1}\circ\cdots\circ F_{\omega_n}(y)
\]
exists and is independent of \(y\in X\). Moreover, the map
\[
        \pi:\partial_{\mathrm{nd}}T_\ast\to X
\]
is continuous and satisfies
\[
        \pi(j\omega)=F_j(\pi(\omega)),
        \qquad j=0,1.
\]
\end{lemma}

\begin{proof} We can assume that $\Lip(F_0\circ F_1)\le c<1$ and that $\diam (F_1\circ F_0)(X)=D<\infty$, the other cases being dealt with the same way.
For a finite word \(w=\ast e_1\dots e_n\), set
\[
        F_w=F_{e_1}\circ\cdots\circ F_{e_n}.
\]
A $01$-transition in \(w\) is a pair of consecutive indices \((k,k+1)\), with
\(k\ge1\), such that
\[
        e_k=0, e_{k+1}=1,
\]
and a $10$-transition is similarly defined. 
We denote by \(N(w)\) the number of $01$-transitions in \(w\) which occur after the first $10$-transition. If $N(w)\ge1$,
Since $F_0$ and $F_1$ are non-expansive, 
\begin{equation}\label{eq:diam-transition-estimate-revised}
        \diam F_w(X)\le c^{N(w)-1}D:
\end{equation}
The first $10$-transition makes the diameter no more that $D$, we do not count the first $01$-transition because we might have a sequence $101$, but all subsequent $01$-transitions do not have digits in common, hence each of them contributes a factor $c$.

Now fix
\[
        \omega=(\ast,e_1,e_2,\ldots)\in\partial_{\mathrm{nd}}T_\ast
\]
and set
\[
        w_n=\ast e_1\dots e_n.
\]
Since \(\omega\) is not eventually constant, the number \(N(w_n)\) tends
to infinity. Hence, by
\eqref{eq:diam-transition-estimate-revised},
\[
        \diam F_{w_n}(X)\longrightarrow0.
\]

Fix \(x\in X\) and put
\[
        x_n=F_{w_n}(x).
\]
If \(m\ge n\), then
\[
        F_{w_m}(X)\subseteq F_{w_n}(X),
\]
because \(F_{w_m}\) factors through \(F_{w_n}\). Hence
\(x_m,x_n\in F_{w_n}(X)\), and therefore
\[
        d(x_m,x_n)\le \diam F_{w_n}(X)\longrightarrow0.
\]
Thus \((x_n)\) is Cauchy. Since \(X\) is complete, \((x_n)\) converges.

Moreover, if \(x,y\in X\), then
\[
        d(F_{w_n}(x),F_{w_n}(y))
        \le \diam F_{w_n}(X)\longrightarrow0.
\]
Thus the limit does not depend on the initial point. We define
\[
        \pi(\omega):=\lim_{n\to\infty}F_{w_n}(x),
\]
where \(x\in X\) is arbitrary.

Since each \(F_j\) is continuous, being non-expansive, we have
\[
\begin{aligned}
F_j(\pi(\omega))
&=
F_j\left(\lim_{n\to\infty}F_{w_n}(x)\right)  \\
&=
\lim_{n\to\infty}F_j(F_{w_n}(x))              \\
&=
\lim_{n\to\infty}F_{jw_n}(x)
=
\pi(j\omega).
\end{aligned}
\]

It remains to prove continuity. Fix
\(\omega\in\partial_{\mathrm{nd}}T_\ast\) and \(\varepsilon>0\). Choose
\(n\) such that
\[
        \diam F_{\omega_1\ldots\omega_n}(X)<\varepsilon.
\]
If \(\eta\in\partial_{\mathrm{nd}}T_\ast\) agrees with \(\omega\) in the
first \(n\) digits, then both \(\pi(\omega)\) and \(\pi(\eta)\) belong
to the closure of \(F_{\omega_1\ldots\omega_n}(X)\). Since the closure
has the same diameter,
\[
        d(\pi(\omega),\pi(\eta))
        \le
        \diam F_{\omega_1\ldots\omega_n}(X)
        <\varepsilon.
\]
This proves continuity.
\end{proof}

\begin{corollary}\label{corTransition}
With the hypotheses of Lemma~\ref{lemmaTransitionRevised}, assume
moreover that \(X=[0,1]\), that \(F_0,F_1\) are strictly increasing,
and that
\[
        F_0(1)=F_1(0).
\]
Then \(\pi\) is strictly increasing on the non-dyadic part of
\([0,1]\), with respect to the usual order.
\end{corollary}

\begin{proof}
It suffices to show that \(\pi(x)<\pi(y)\) whenever
\[
        x=\sum_{n=1}^\infty \frac{d_n}{2^n}
        <
        y=\sum_{n=1}^\infty \frac{e_n}{2^n}
\]
are non-dyadic points. Let \(m\ge0\) be the first place at which the
binary expansions differ, so that
\[
        d_1=e_1,\ldots,d_m=e_m,
        \qquad
        d_{m+1}=0,
        \qquad
        e_{m+1}=1.
\]
Since the finite composition
\[
        F_{d_1}\circ\cdots\circ F_{d_m}
\]
is strictly increasing, it is enough to consider the case \(m=0\).

Thus \(x\) begins with \(0\), while \(y\) begins with \(1\). Since \(x\)
is not dyadic, its binary expansion is not eventually equal to \(1\).
Hence there exists \(k\ge2\) such that
\[
        d_2=\cdots=d_{k-1}=1,
        \qquad
        d_k=0.
\]
Define, for non-dyadic binary expansions,
\[
        \check\sigma\left(\sum_{n=1}^\infty \frac{d_n}{2^n}\right)
        =
        \sum_{n=1}^\infty \frac{d_{n+1}}{2^n}.
\]
Then \(\check\sigma\) is the symbolic shift, written in real
coordinates. By the functional relation for \(\pi\),
\[
        \pi(x)
        =
        (F_0\circ F_1^{k-2}\circ F_0)
        \bigl(\pi(\check\sigma^k x)\bigr).
\]
Since \(\pi(\check\sigma^k x)\le1\), and since the maps are increasing,
\[
\begin{aligned}
\pi(x)
&\le
F_0\left(F_1^{k-2}(F_0(1))\right)  \\
&=
F_0\left(F_1^{k-2}(F_1(0))\right)  \\
&=
F_0(F_1^{k-1}(0)).
\end{aligned}
\]
Since \(F_1\) is strictly increasing and \(F_1(1)=1\), we have
\[
        F_1^{k-1}(0)<1.
\]
Therefore
\[
        \pi(x)
        \le
        F_0(F_1^{k-1}(0))
        <
        F_0(1)
        =
        F_1(0).
\]
On the other hand, since \(y\) begins with \(1\),
\[
        \pi(y)=F_1(\pi(\check\sigma y))\ge F_1(0).
\]
Hence
\[
        \pi(x)<\pi(y).
\]
\end{proof}
We will use the Schwarz-Pick lemma and some of its consequences concerning the hyperbolic metric in the complex unit disc, which can be found in the first chapter of \cite{Ah1973}.
\begin{proof}[Proof of Theorem \ref{theodeRhamExists}] 
We indicate how Lemma~\ref{lemmaTransitionRevised} applies to the
fractional-linear system in Theorem~\ref{theodeRhamExists}.

For comparison with the Schwarz--Pick lemma, conjugate \([0,1]\) to
\([-1,1]\) by
\[
        \varphi(z)=2z-1.
\]
Set
\[
        G_j=\varphi\circ F_j\circ\varphi^{-1},
        \qquad j=0,1.
\]
Then \(G_0,G_1\) are increasing fractional-linear transformations of
\([-1,1]\) into itself, satisfying
\[
        G_0(-1)=-1,
        \qquad
        G_1(1)=1,
        \qquad
        G_0(1)=G_1(-1)\in(-1,1).
\]
Moreover,
\[
        G_j([-1,1])\subseteq[-1,1],
        \qquad j=0,1.
\]

The maps \(G_j\) extend to fractional-linear self-maps of the unit disc
\[
        \mathbb D=\{w\in\mathbb C: |w|<1\}.
\]
Since fractional-linear maps send straight lines and circles to straight lines and
circles, the mixed maps \(G_0\circ G_1\) and \(G_1\circ G_0\) map
\(\mathbb D\) into discs whose closures are compactly contained in
\(\mathbb D\). 
Indeed, a real fractional-linear map sending \([-1,1]\) onto
\([a,b]\subset[-1,1]\) maps \(\partial\mathbb D\) onto the circle with
diameter \([a,b]\), hence maps \(\mathbb D\) into itself. For the mixed
maps the corresponding intervals are compactly contained in \((-1,1)\),
so their image disks are compactly contained in \(\mathbb D\).
Thus there exists a compact set \(K\Subset\mathbb D\)
such that
\[
        (G_0\circ G_1)(\mathbb D)
        \cup
        (G_1\circ G_0)(\mathbb D)
        \subseteq K.
\]

Let \(d_{\mathbb D}\) be the hyperbolic distance associated with
\[
        ds^2=\frac{4|dw|^2}{(1-|w|^2)^2}.
\]
By the Schwarz--Pick lemma, \(G_0\) and \(G_1\) are non-expansive for
\(d_{\mathbb D}\). Since the mixed images are contained in
\(K\Subset\mathbb D\), the equality case in Schwarz--Pick gives a
strict contraction estimate: there exists \(0<c<1\) such that
\[
        d_{\mathbb D}((G_0\circ G_1)(w_1),(G_0\circ G_1)(w_2))
        \le
        c\,d_{\mathbb D}(w_1,w_2),
\]
and
\[
        d_{\mathbb D}((G_1\circ G_0)(w_1),(G_1\circ G_0)(w_2))
        \le
        c\,d_{\mathbb D}(w_1,w_2).
\]
After conjugating back to \([0,1]\), Lemma~\ref{lemmaTransitionRevised}
and Corollary~\ref{corTransition} give a strictly increasing
continuous solution on the non-dyadic symbolic part.

We also record uniqueness. Let \(h\) be any strictly increasing,
left-continuous solution of \eqref{eqDR}. If \(x\in[0,1]\) is not
dyadic and has binary expansion
\[
        x=0.e_1e_2e_3\ldots,
\]
then, by iterating the functional equation, for every \(n\) we may write
\[
        h(x)
        =
        F_{e_1}\circ\cdots\circ F_{e_n}(y_n)
\]
for some \(y_n\in[0,1]\). By the transition lemma, the limit of the
right-hand side is independent of the choice of \(y_n\). Hence \(h(x)\)
coincides with the function constructed above at every non-dyadic point.
Since both functions are left continuous, they coincide everywhere on
\([0,1]\).

It remains to determine when the left-continuous solution is continuous
at the dyadic points. Every dyadic point different from \(0\) and \(1\)
has the form
\[
        g_w(1/2)
\]
for some finite word \(w\). The two one-sided limits of \(f\) at
\(g_w(1/2)\) are obtained by applying the strictly increasing map
\(F_w\) to the two one-sided limits at \(1/2\). Hence it suffices to
decide continuity at \(1/2\).

Writing \(A_0,A_1\) in the linear parameters of \eqref{eqLinPar}, one
obtains, by induction,
\begin{equation}\label{eqFzeroItRevised}
        \zeta_0
        :=
        \lim_{n\to\infty}F_0^n(x)
        =
        \begin{cases}
            0,
            & \lambda\ge1,\\[0.4em]
            \dfrac{1-\lambda}{\gamma-\lambda},
            & 0<\lambda<1.
        \end{cases}
\end{equation}
Similarly,
\begin{equation}\label{eqFoneItRevised}
        \zeta_1
        :=
        \lim_{n\to\infty}F_1^n(x)
        =
        \begin{cases}
            1,
            & \gamma\le q+1,\\[0.4em]
            \dfrac{1}{\gamma-q},
            & \gamma>q+1.
        \end{cases}
\end{equation}
In the latter case, the limit can be verified by considering that the graph of \(F_1\) crosses the diagonal at
\[
        x=\frac{1}{\gamma-q},
\]
and by drawing cobweb diagrams.

The solution is continuous at \(1/2\) if and only if the two one-sided
limits agree, that is,
\[
        F_1(\zeta_0)=F_0(\zeta_1).
\]
By conditions (A1) and (A2), this happens if and only if
\[
        \zeta_0=0
        \qquad\text{and}\qquad
        \zeta_1=1.
\]
Equivalently,
\[
        \lambda\ge1,
        \qquad
        q\ge\gamma-1.
\]
These are precisely the parameters in
\[
        \mathcal L_{\mathrm{tr}}
        =
        \{(\lambda,\gamma,q):\lambda\ge1,\ q\ge\gamma-1\}.
\]
\end{proof}

\subsection{Dyadic potential theory and the proof of Theorem \ref{theodrCap}}

The definition of capacity given in \eqref{eq:cap-def} is a special case of the axiomatic definition one finds in \cite{AH1996} \S2.3. For $a\in E(T_\ast)$, it will be convenient to denote by $T_a$ the subtree of $T_\ast$ having $a$ as root-edge: the edges of $T_a$ have the form $ae_1\dots e_m$ with $e_1,\dots, e_m\in\{0,1\}$ (we might think of this in terms of back-shifts on $T_\ast$). Consider the kernel
\[
E(T_\ast)\times\partial T_\ast\xrightarrow{\kappa_\ast}[0,\infty],\ \kappa_\ast(a,\zeta)=\chi_{P_\ast(\zeta)}(a)=\chi_{\partial T_a}(\zeta)=\begin{cases}
    1\text{ if }a\in P_\ast(\zeta),\\
    0\text{ if }a\notin P_\ast(\zeta)
\end{cases}.
\]
We consider $\partial T_\ast$ as a topological space, and we consider $E(T_\ast)$ endowed with the measure $\mathrm{w}$. 
For $E(T_\ast)\xrightarrow{\varphi}(0,\infty)$, we let $\kappa_\ast\varphi(\zeta)=\sum_{a\in P_\ast(\zeta)}\varphi(a)\mathrm{w}(a)$.
Then, our definition might be phrased as
\[
\Cap_\ast^{\mathrm{w}}(E)=\inf\{\|\varphi\|_{L^2(\mathrm{w})}^2:\varphi\ge0,\kappa_\ast\varphi\ge1\text{ on }E \}.
\]
For a positive Borel measure $\mu$ on $\partial T_\ast$ and $a\in E(T_\ast)$, let
\[
\check{\kappa}_\ast\mu(a):=\int_{\partial T_\ast}\kappa_\ast(a,\zeta)d\mu(\zeta)=\mu(\partial T_a).
\]
In \S2.5 of \cite{AH1996} it is proved that, if $E\subseteq\partial T_\ast$ is a Borel set, then we have a dual characterization of $\Cap_\ast^{\mathrm{w}}(E)$,
\begin{equation}\label{eqCapDual}
    \Cap_\ast^{\mathrm{w}}(E)=\sup_{\text{supp}(\mu)\subseteq E}\frac{\mu(E)^2}{\mathcal{E}_\ast^{\mathrm{w}}(\mu)}=\sup_{\text{supp}(\nu)\subseteq E,\nu(E)=1}\frac{1}{\mathcal{E}_\ast^{\mathrm{w}}(\nu)},
\end{equation}
where 
\[
\mathcal{E}_\ast^{\mathrm{w}}(\mu)=\sum_{a\in E(T_\ast)}(\check{\kappa}_\ast\mu)^2(a)\mathrm{w}(a)=\sum_{a\in E(T_\ast)}\mu(\partial T_a)^2\mathrm{w}(a).
\]
is the {\it $\mathrm{w}$-energy} of $\mu$.

If $a$ is an edge of $T_\ast$ and $E\subseteq\partial T_a\subseteq\partial T_\ast$, we can define its $\mathrm{w}$-capacity $\Cap_a^{\mathrm{w}}(E)$ in $T_a$ by restricting the weight $\mathrm{w}$ to $E(T_a)$. We can identify $\partial T_a$ with a subset of $\partial T_\ast$.  A peculiar feature of dyadic potential theory is that a recursive formula allows to reconstruct $\Cap_a^\mathrm{w}(E)$ from $\Cap_{a0}^\mathrm{w}(E\cap\partial T_{a0})$ and $\Cap_{a1}^\mathrm{w}(E\cap\partial T_{a1})$,
\begin{equation}\label{eqCapRecursion}
    \Cap^\wt_{\mathrm{a}}(E)=\frac{\Cap^\wt_{\mathrm{a0}}(E\cap\partial T_{a0})+\Cap^\wt_{\mathrm{a1}}(E\cap\partial T_{a1})}{1+\wt(a)[\Cap^\wt_{\mathrm{a0}}(E\cap\partial T_{a0})+\Cap^\wt_{\mathrm{a1}}(E\cap\partial T_{a1})]}.
\end{equation}
See e.g. Theorem 30 in \cite{ARSW2014}, for a general formulation including nonlinear potentials, but also \cite[Chapter 16]{LP2016}, or \cite[Chapter V, \S6, "Capacity of the boundary of a tree"]{S1994}. We mention here that \eqref{eqCapRecursion} provides an algorithm to compute the dyadic capacity of a set in terms of a {\it branched continued fraction}. To the best of my knowledge, the relation between branched continued
fractions, potential theory on trees, and iterated function systems has
not been systematically explored. It would be interesting to investigate
this connection further. 
See the recent survey \cite{BoK2018} and its bibliography.

We will also use that capacities scale with the weight in an obvious way,
\begin{equation}\label{eqReascale}
        \Cap_\ast^{c \mathrm{w}}(E)=c^{-1}\Cap_\ast^{\mathrm{w}}(E).
\end{equation}
Consider the back-shifts $E(T_\ast)\xrightarrow{\tau_j}E(T_{j})$, $\tau_j(\ast  e_1\dots e_n)=(\ast je_1\dots e_n)$. We have that
\begin{equation}\label{eqBScomm}
    g_j\circ\Lambda=\Lambda\circ \tau_j.
\end{equation}
The recursive relation \eqref{eqCapRecursion} implies that for Borel subsets $E,F\subseteq\partial T_\ast$, $\lambda,\mu>0$, and $\wt=\wt^{\lambda,\mu}$, we have the following (the reader is urged to draw a picture to verify  the obvious geometric content of the calculation):
\begin{equation}\label{eqRicorso}
\begin{aligned}
\Cap^{\mathrm{w}^{\lambda,\mu}}_\ast(\tau_0(E)\cup\tau_1(F))&=\frac{\Cap_{\ast0}^{\mathrm{w}^{\lambda,\mu}}(\tau_0(E))+\Cap_{\ast1}^{\mathrm{w}^{\lambda,\mu}}(\tau_1(F))}{1+\Cap_{\ast0}^{\mathrm{w}^{\lambda,\mu}}(\tau_0(E))+\Cap_{\ast1}^{\mathrm{w}^{\lambda,\mu}}(\tau_1(F))}\\
&=\frac{\frac{\Cap_{\ast0}^{\frac{\mathrm{w}^{\lambda,\mu}}{\lambda}}(\tau_0(E))}{\lambda}+\frac{\Cap_{\ast1}^{\frac{\mathrm{w}^{\lambda,\mu}}{\mu}}(\tau_1(F))}{\mu}}{1+\frac{\Cap_{\ast0}^{\frac{\mathrm{w}^{\lambda,\mu}}{\lambda}}(\tau_0(E))}{\lambda}+\frac{\Cap_{\ast1}^{\frac{\mathrm{w}^{\lambda,\mu}}{\mu}}(\tau_1(F))}{\mu}}\\
&=\frac{\frac{\Cap_\ast^{\mathrm{w}^{\lambda,\mu}}(E)}{\lambda}+\frac{\Cap_\ast^{\mathrm{w}^{\lambda,\mu}}(F)}{\mu}}{1+\frac{\Cap_\ast^{\mathrm{w}^{\lambda,\mu}}(E)}{\lambda}+\frac{\Cap_\ast^{\mathrm{w}^{\lambda,\mu}}(F)}{\mu}}.
\end{aligned}
\end{equation}
\begin{proposition}\label{propTotCap}
    Let $\lambda,\mu>0$. Then, $\Cap_\ast^{\lambda,\mu}(\partial T_\ast)$ is given by \eqref{eqTotCap},
    \[
    C:=\Cap_\ast^{\lambda,\mu}(\partial T_\ast)=\begin{cases}
        \frac{\lambda+\mu-\lambda\mu}{\lambda+\mu}\text{ if }\lambda+\mu-\lambda\mu>0,\\
        0\text{ if }\lambda+\mu-\lambda\mu\le0.\end{cases}
    \]
\end{proposition}
\begin{proof}
From \eqref{eqRicorso}, since $\partial T_\ast=\tau_0(\partial T_\ast)\cup\tau_1(\partial T_\ast)$:
\[
C:=\Cap_\ast^{\lambda,\mu}(\partial T_\ast)=\frac{\frac{\Cap_\ast^{\lambda,\mu}(\partial T_\ast)}{\lambda}+\frac{\Cap_\ast^{\lambda,\mu}(\partial T_\ast)}{\mu}}{1+\frac{\Cap_\ast^{\lambda,\mu}(\partial T_\ast)}{\lambda}+\frac{\Cap_\ast^{\lambda,\mu}(\partial T_\ast)}{\mu}}=\frac{C\left(\frac{1}{\lambda}+\frac{1}{\mu}\right)}{1+C\left(\frac{1}{\lambda}+\frac{1}{\mu}\right)}.
\]
The equation has solutions $C_1=0$ and $C_2=\frac{\lambda+\mu-\lambda\mu}{\lambda+\mu}$. Since we a priori know that $C\ge0$, when $C_2\le0$ the whole boundary has zero capacity. 

We now show that, when $C_2>0$, $\Cap_\ast^{\lambda,\mu}(\partial T_\ast)=C_2$. Let $\rho_0=\frac{\frac{1}{\lambda}}{\frac{1}{\lambda}+\frac{1}{\mu}}$ and $\rho_1=\frac{\frac{1}{\mu}}{\frac{1}{\lambda}+\frac{1}{\mu}}$, and consider the probability Bernoulli measure $\nu$ on $\partial T_\ast$ defined on cylinder sets by
\[
\nu(\partial T_a)=\rho_0^{|a|_0}\rho_1^{|a|_1}.
\]
Using the fact that $\lambda\mu<\lambda+\mu$, we can compute its energy,
\begin{eqnarray*}
    \mathcal{E}^{\lambda,\mu}_\ast(\nu)&=&\sum_{a\in E(T_\ast)}\nu(\partial T_a)^2\wt^{\lambda,\mu}(a)\\
    &=&\sum_{n=0}^\infty\sum_{a:|a|=n}(\rho_0^{|a|_0}\rho_1^{|a|_1})^2\lambda^{|a|_0}\mu^{|a|_1}\\
    &=&\sum_{n=0}^\infty\sum_{k=0}^n\binom{n}{k}\frac{\mu^{2n-2k}\lambda^{2k}\lambda^{n-k}\mu^{k}}
    {(\lambda+\mu)^{2n}}\\
    &=&\sum_{n=0}^\infty\left(\frac{\lambda\mu}{\lambda+\mu}\right)^n\\
    &=&\frac{\lambda+\mu}{\lambda+\mu-\lambda\mu}\\
    &=&C^{-1}<\infty.
\end{eqnarray*}
By the dual characterization \eqref{eqCapDual}, this implies
\[
\Cap_\ast^{\lambda,\mu}(\partial T_\ast)
\ge
\frac{1}{\mathcal E_\ast^{\lambda,\mu}(\nu)}
=
C_2
>0.
\]
Since the recursive identity showed that the only possible values are
\(0\) and \(C_2\), it follows that
\[
\Cap_\ast^{\lambda,\mu}(\partial T_\ast)=C_2.
\]
\end{proof}
Incidentally, this calculation shows that the Bernoulli measure $\nu=\nu^{\lambda,\mu}$ realizes the supremum in dual definition of capacity \eqref{eqCapDual}, i.e. that $C\nu^{\lambda,\mu}$ is the {\it equilibrium measure} for $\partial T_\ast$ with respect to the weight $\wt^{\lambda,\mu}$. 

\begin{proof}[Proof of Theorem \ref{theodrCap}]
Let $c(x)=\Cap_\ast^{\lambda,\mu}(\Lambda^{-1}([0,x)))$, $[0,1]\xrightarrow{c}[0,C]$, which is an increasing function of $x$.
We first verify that $c$ is left continuous. The identity $c(1)=C$ will be proved at the end of the argument. A property of any capacity $\Cap$ is that, if $A_n\nearrow A$, then $\Cap(A_n)\nearrow\Cap(A)$.
Then, if $x_n\nearrow x$,
\begin{eqnarray*}
\lim_{n\to\infty}c(x_n)&=&\sup_{n\to\infty}\Cap^{\lambda,\mu}(\Lambda^{-1}([0,x_n)))=\Cap^{\lambda,\mu}\left(\bigcup_{n=1}^\infty\Lambda^{-1}([0,x_n))\right)\\
    &=&\Cap^{\lambda,\mu}(\Lambda^{-1}([0,x)))=c(x).
\end{eqnarray*}
It is easily verified that for $j=0,1$
\[
\Lambda^{-1}([g_j(0),g_j(x)))=\tau_j(\Lambda^{-1}([0,x))),
\]
a fact that we will use below.
 Then,
\begin{eqnarray*}
    c(g_0(x))
    &=&
    \Cap_\ast^{\lambda,\mu}(\Lambda^{-1}([0,g_0(x))))\\
    &=&
    \Cap_\ast^{\lambda,\mu}(\Lambda^{-1}([g_0(0),g_0(x))))\\
    &=&
    \Cap_\ast^{\lambda,\mu}(\tau_0(\Lambda^{-1}([0,x))))\\
    &=&
    \frac{\frac{c(x)}{\lambda}}{1+\frac{c(x)}{\lambda}}
\end{eqnarray*}
by \eqref{eqRicorso}. Similarly,
\begin{eqnarray*}
    c(g_1(x))
    &=&
    \Cap_\ast^{\lambda,\mu}(\Lambda^{-1}([0,g_1(x))))\\
    &=&
    \Cap_\ast^{\lambda,\mu}
    \left(
    \Lambda^{-1}\bigl([g_0(0),g_0(1))\cup[g_1(0),g_1(x))\bigr)
    \right)\\
    &=&
    \Cap_\ast^{\lambda,\mu}
    \left(
    \tau_0(\Lambda^{-1}([0,1)))
    \cup
    \tau_1(\Lambda^{-1}([0,x)))
    \right)\\
    &=&
    \frac{\frac{c(1)}{\lambda}+\frac{c(x)}{\mu}}
    {1+\frac{c(1)}{\lambda}+\frac{c(x)}{\mu}},
\end{eqnarray*}
again by \eqref{eqRicorso}. Then substitute $f(x)=C^{-1}c(x)$.

Finally, we have to show that $f(1)=1$, or, equivalently, that 
\[
        c(1):=\Cap_\ast^{\lambda,\mu}\bigl(\Lambda^{-1}([0,1))\bigr)=C.
\]
Let \(x_n\uparrow1\), \(x_n<1\), and set
\[
        E_n=\Lambda^{-1}([0,x_n)),
\]
and let \(\nu\) be the Bernoulli equilibrium measure from the
proof of Proposition~\ref{propTotCap}. Then
\[
        \mathcal E_\ast^{\lambda,\mu}(\nu)=C^{-1}.
\]
Moreover \(\nu(E_n)\uparrow1\). Put
\[
        \nu_n=\frac{\nu|_{E_n}}{\nu(E_n)}.
\]
Then \(\nu_n\) is a probability measure supported on \(E_n\), and
\[
        \mathcal E_\ast^{\lambda,\mu}(\nu_n)
        =
        \nu(E_n)^{-2}
        \mathcal E_\ast^{\lambda,\mu}(\nu|_{E_n}).
\]
By monotone convergence,
\[
        \mathcal E_\ast^{\lambda,\mu}(\nu|_{E_n})
        \longrightarrow
        \mathcal E_\ast^{\lambda,\mu}(\nu)
        =
        C^{-1}.
\]
Since \(\nu(E_n)\to1\), we get
\[
        \mathcal E_\ast^{\lambda,\mu}(\nu_n)
        \longrightarrow C^{-1}.
\]
The dual formula gives
\[
        \Cap_\ast^{\lambda,\mu}(E_n)
        \ge
        \frac{1}{\mathcal E_\ast^{\lambda,\mu}(\nu_n)}
        \longrightarrow C.
\]
Therefore
\[
        \Cap_\ast^{\lambda,\mu}\bigl(\Lambda^{-1}([0,1))\bigr)=C.
\]

\end{proof}
\begin{proposition}\label{propPTpar} The trapping condition (AA3) holds if and only if \eqref{eqPTparTrap} holds,
\[
\mathcal{P}_{\mathrm{tr}}=\{(\lambda,\mu,t)\in\mathcal{P}:\lambda\ge1,(\lambda+\mu)^2\ge\lambda^2\mu\}.
\]
\end{proposition}
The proof consists in computing the dynamical invariants in terms of the potential theoretic parameters.
\begin{proof}
Let $A_0,A_1$ be the matrices in \eqref{eqAzoeruno}, and let $B_j=T_t A_j T_t^{-1}$.

One computes
\begin{equation}\label{eq:B0}
 B_0=
 \begin{pmatrix}
 1 & 0\\
 tC+(1-t)(1-\lambda) & \lambda
 \end{pmatrix}
\end{equation}
and
\begin{equation}\label{eq:B1}
B_1=
\begin{pmatrix}
\frac{1}{\mu}+\frac{1}{\lambda}-\frac{1}{\lambda t} & \frac{1}{\lambda t}\\[0.5em]
\frac{(1-t)+tC}{\mu}-(1-t)-\frac{(1-t)C}{\lambda}-\frac{(1-t)^2}{\lambda t}
& 1+\frac{C}{\lambda}+\frac{1-t}{\lambda t}
\end{pmatrix}.
\end{equation}
The calculation of the dynamical parameters for the data $\Phi(B_0;\cdot),\Phi(B_1;\cdot)$ is routine:
\begin{align}\label{eqProjParUno}
    \gamma&=1+t\frac{\lambda^2}{\lambda+\mu};\\
    \label{eqProjParDue} r_0&=\begin{cases}
        -1+
t\frac{\lambda^2}{(\lambda+\mu)(\lambda-1)}
\text{ if } \lambda\neq1,\\
+\infty\text{ if }\lambda=1;
    \end{cases}\\
\label{eqProjParTre}    r_1&=
-1+
t
\frac{(\lambda+\mu)^2-\lambda^2\mu}
{\mu(\lambda+\mu)}.
\end{align}
Since $\alpha=\min\{0,r_0,r_1\}$, the proposition follows immediately.    
\end{proof}
The curves $\Gamma_1=\{\lambda=1\}$, $\Gamma_2=\{\lambda+\mu=\lambda\mu\}$, and $\Gamma_3=\{(\lambda+\mu)^2=\lambda^2\mu\}$ do not meet in the quadrant $\lambda>0,\mu>0$, and they divide it into four regions. Each region and each curve corresponds to a distinctive property of the de Rham system in potential theoretic coordinates.

\subsection{Classical vs. dyadic Riesz capacities}\label{SSectRBc}
For $0<\tau\le 1/2$ and a Borel measure $\omega\ge0$ on $[0,1]$, let 
\[
\kappa_\tau\omega(x)=\int_0^1\frac{d\omega(y)}{|x-y|^{1-\tau}},
\]
where $0\le x\le1$, and
\[
\mathcal{E}^\tau(\omega)=\int_0^1[\kappa_\tau\omega(x)]^2dx.
\]
The {\it Riesz capacity} $\Qap_\tau(E)$ of a Borel subset $E$ of $[0,1]$ is
\[
\Qap_\tau(E)=\sup\left\{\frac{\omega^2(E)}{\mathcal{E}^\tau(\omega)}: \omega\ge0,\ \text{supp}(\omega)\subseteq E\right\}.
\]
Equivalently, for \(0<\tau<1/2\), this capacity is comparable to the
capacity associated with the Riesz energy of order \(1-2\tau\), and, for $\tau=1/2$, to the capacity associated with a logarithmic energy,
\[
\mathcal{E}^\tau(\omega)\approx\begin{cases}
    \iint_{[0,1]^2}\frac{d\omega(x)d\omega(y)}{|x-y|^{1-2\tau}}\text{ if }0<\tau<1/2\\
    \iint_{[0,1]^2}\log\frac{2}{|x-y|}d\omega(x)d\omega(y)\text{ if }\tau=1/2.
\end{cases}
\]
It is explicitly proved in \cite{ARSW2014} (Theorem 1), and the result is implicit in \cite{BP1992}, that for $0\le s<1$
\begin{equation}\label{eqCapDnonD}
    \begin{split}
        \Cap_s(\Lambda^{-1}(E))&\approx\Qap_{\frac{1-s}{2}}(E)\text{ for a Borel subset $E$ of $[0,1]$;}\\
        \Qap_{\frac{1-s}{2}}(\Lambda(F))&\approx\Cap_s(F)\text{ for a Borel subset $F$ of $\partial T_\ast$}.
    \end{split}
\end{equation}
The two-sided estimates \eqref{eqCapDnonD}, together with recursive formulas like \eqref{eqCapRecursion}, can be of great help in estimating set capacities in classical potential theory. See \cite{AC2022} for a simple application.
\begin{lemma}\label{lemmaCapDnonD}
    Fix $0\le s<1$. Let $[0,1]\xrightarrow{\psi}[0,1]$ be strictly increasing, $C^1$ and having $C^1$ inverse, and let $E\subseteq[0,1]$ be a Borel set. Then, $\Cap_s(\Lambda^{-1}(E))=0$ if and only if $\Cap_s(\Lambda^{-1}(\psi(E)))=0$. Also, for a Borel subset $F$ of $\partial T_\ast$, we have $\Cap_s(F)=0$ if and only if $\Cap_s(\Lambda^{-1}(\psi(\Lambda(F))))=0$.
\end{lemma}
\begin{proof} The map $\Lambda$ maps Borel sets to Borel sets \cite[Lemma 15]{ARSW2014}, hence we have no problem with capacitability. Since $|\psi(x)-\psi(y)|\approx|x-y|$ and $d\psi(x)=\psi'(x)dx\approx dx$, we have that
     $\Qap_{\frac{1-s}{2}}(E)\approx \Qap_{\frac{1-s}{2}}(\psi(E))$, with multiplicative constants which only depend on $s$ and $\psi$. We can then use \eqref{eqCapDnonD}.
\end{proof}
We recall the standard relation between Riesz capacities and
Hausdorff dimension; see, for instance, \cite[Chapter 5]{AH1996} or \cite[Theorem 6.4]{Falconer1986}.
For \(0<\tau<1/2\), the critical Hausdorff exponent associated with
\(\Qap_\tau\) is
\[
        d_\tau:=1-2\tau .
\]
The endpoint case \(\tau=1/2\) corresponds to logarithmic capacity and
has critical exponent \(0\).

For every Borel set \(E\subseteq[0,1]\),
\[
        \dim_H E<d_\tau
        \quad\Longrightarrow\quad
        \Qap_\tau(E)=0,
\]
whereas
\[
        \dim_H E>d_\tau
        \quad\Longrightarrow\quad
        \Qap_\tau(E)>0.
\]
By \eqref{eqCapDnonD},
\[
        \dim_H E<s
        \quad\Longrightarrow\quad
        \Cap_s(\Lambda^{-1}(E))=0,
\]
while
\[
        \dim_H E>s
        \quad\Longrightarrow\quad
        \Cap_s(\Lambda^{-1}(E))>0.
\]
Moreover,
\begin{equation}\label{eqMissing}
        \mathcal H^s(E)<\infty
        \quad\Longrightarrow\quad
        \Cap_s(\Lambda^{-1}(E))=0.
\end{equation}
Thus the dyadic capacity \(\Cap_s\) refines Hausdorff dimension at the
critical exponent \(s\).

\subsection{The proof of Theorem \ref{theoMain}}
The following elementary martingale lemma strengthens, in the direction
needed here, the estimate $M_n/n\to0$ used in
\cite[Lemma~2.3(2)]{Okamura2014}.  The latter estimate controls only the
linear growth of the martingale term.  By contrast, the lemma below shows
that a martingale whose negative increments are uniformly bounded cannot
escape to $+\infty$; equivalently, it returns infinitely often below some
finite level.  This recurrence statement is what allows us to obtain
vanishing capacity at the critical exponent.
The proof could be easily carved out of that of \cite[Theorem~4.3.1, p.~224]{Durrett2019}. We include it here for convenience of the reader.

\begin{lemma}[Martingale lemma]
\label{lem:no-plus-infty}
Let $(M_n,\mathcal F_n)_{n\geq0}$ be a real-valued martingale with $M_0=0$.
Assume that there is a constant $B<\infty$ such that
\[
        M_n-M_{n-1}\geq-B
        \qquad\text{a.s. for every }n\geq1.
\]
Then
\[
        \mathbb P\left(\liminf_{n\to\infty}M_n=+\infty\right)=0.
\]
Equivalently,
\[
        \mathbb P\left(
        \bigcup_{m=1}^\infty
        \{M_n\leq m\ \text{for infinitely many }n\}
        \right)=1.
\]
\end{lemma}

\begin{proof}
For an integer $a\geq1$, define the stopping time
\[
        \tau_a:=\inf\{n\geq0:M_n\leq-a\},
\]
with the convention $\inf\varnothing=\infty$.  Since the negative
increments are bounded below by $-B$, at the first entrance below $-a$ we
have
\[
        M_{\tau_a}\geq-a-B.
\]
Therefore the stopped martingale
\[
        N_n:=M_{n\wedge\tau_a}+a+B
\]
is nonnegative.  Hence $N_n$ converges almost surely to a finite limit.  On
the event $\{\tau_a=\infty\}$, the original martingale
$M_n=N_n-a-B$ consequently converges to a finite limit and cannot tend to
$+\infty$.

If $M_n\to+\infty$, then the numerical sequence $(M_n)$ has a finite lower
bound.  Hence there is an integer $a\geq1$ such that $M_n>-a$ for every
$n$, that is, $\tau_a=\infty$.  It follows that
\[
        \{M_n\to+\infty\}
        \subseteq
        \bigcup_{a=1}^\infty
        \bigl(\{\tau_a=\infty\}\cap\{M_n\to+\infty\}\bigr).
\]
Each event on the right has probability zero by the preceding paragraph.
Thus
\[
        \mathbb P(M_n\to+\infty)=0.
\]
Since
\[
        \liminf_{n\to\infty}M_n=+\infty
        \quad\Longleftrightarrow\quad
        M_n\to+\infty,
\]
the first assertion follows.

If the displayed union in the second assertion failed at a point, then,
for every integer $m$, the inequality $M_n\leq m$ would hold only finitely
often.  Hence $M_n\to+\infty$, and the second assertion follows from the
first.
\end{proof}

Below, $[0,1]$ is endowed with its Borel $\sigma$-algebra and with the
filtration $(\mathcal F_n)_{n\geq0}$ generated by the dyadic intervals of
generation $n$.  We denote by $I_n(x)$ the unique generator of
$\mathcal F_n$ containing $x$.

\begin{lemma}[Geometric lemma]
\label{lem:stopping-cylinder}
Let $\nu$ be a Borel probability measure on $[0,1]$, and let
\[
        R_n(x)=\nu(I_n(x)).
\]
Suppose that
\[
        -\log R_n(x)=L_n(x)+M_n(x),
\]
where $M_n$ is $\mathcal F_n$-measurable and
\[
        L_n(x)\leq n\theta
\]
for every $n$ and almost every $x$.  Put
\[
        d=\frac{\theta}{\log2}.
\]
For $m\in\mathbb R$, define
\[
        E_m:=\{x:M_n(x)\leq m\text{ for infinitely many }n\}.
\]
Then
\[
        \mathcal H^d(E_m)\leq e^m.
\]
In particular, if $0\leq d<1$, then
\[
        \Cap_d\bigl(\Lambda^{-1}(E_m)\bigr)=0.
\]
\end{lemma}

\begin{proof}
Fix $m\in\mathbb R$ and an integer $N\geq1$.  Define the first-entry time
\[
        \tau_{m,N}(x):=\inf\{n\geq N:M_n(x)\leq m\}
\]
and set
\[
        G_{m,N}:=\{x:\tau_{m,N}(x)<\infty\}.
\]
Since $M_n$ is $\mathcal F_n$-measurable, $\{\tau_{m,N}=n\}$ is a union of
dyadic intervals of generation $n$.

Let $\mathcal C_{m,N}$ be the family of all dyadic intervals $I$ on which
$\tau_{m,N}$ is equal to the generation of $I$.  These first-entry
cylinders cover $G_{m,N}$, have diameter at most $2^{-N}$, and are pairwise
disjoint.  Indeed, two dyadic intervals are either disjoint or one contains
the other, and strict containment of two first-entry cylinders would
contradict the definition of the first-entry time.

Let $I\in\mathcal C_{m,N}$ have generation $n$.  On $I$, both $R_n$ and
$M_n$ are constant, and $M_n\leq m$.  Therefore
\[
        -\log\nu(I)=-\log R_n=L_n+M_n\leq n\theta+m.
\]
It follows that
\[
        \nu(I)\geq e^{-m}e^{-n\theta}=e^{-m}|I|^d,
\]
or equivalently
\[
        |I|^d\leq e^m\nu(I).
\]
Since the intervals in $\mathcal C_{m,N}$ are pairwise disjoint,
\[
        \sum_{I\in\mathcal C_{m,N}}|I|^d
        \leq e^m\sum_{I\in\mathcal C_{m,N}}\nu(I)
        \leq e^m.
\]

Now $E_m\subseteq G_{m,N}$ for every $N$.  Hence
\[
        \mathcal H^d_{2^{-N}}(E_m)\leq e^m.
\]
Letting $N\to\infty$ gives
\[
        \mathcal H^d(E_m)\leq e^m.
\]
The capacitary conclusion follows from \eqref{eqMissing}.
\end{proof}

\begin{proof}[Proof of Theorem~\ref{theoMain}]
Put
\[
        R_n(x)=\mu_f(I_n(x)).
\]
As in \cite{Okamura2014}, define
\[
        L_n
        =
        \sum_{k=1}^n
        \mathbb E^{\mu_f}
        \left[
        \left.
        -\log\frac{R_k}{R_{k-1}}
        \right|
        \mathcal F_{k-1}
        \right],
        \qquad
        M_n=-\log R_n-L_n.
\]
Then $(M_n,\mathcal F_n)$ is a martingale and $M_0=0$.

Although Okamura assumes continuity of $f$, the proof of
\cite[Lemma~2.1(2)]{Okamura2014} uses only the identity
\[
        \mu_f([a,b))=f(b)-f(a)
\]
for half-open dyadic intervals, together with the functional equation at
their endpoints.  Consequently, the same cylinder-ratio formula holds for
our left-continuous solution, also when it has atoms at dyadic points.

The trapping condition (AA3) (see \cite[Lemmas~2.1(2) and~2.3(1)]{Okamura2014}) implies that
\[
        L_n-L_{n-1}=s\bigl(p_0(Z_{n-1})\bigr)
\]
where the random variable $Z_n$ is "trapped" in $[\alpha,\beta]$,
\[
        Z_n\in[\alpha,\beta].
\]
Moreover, he proves that
\[
\begin{aligned}
M_n-M_{n-1}
&=
-\log\frac{R_n}{R_{n-1}}
-
\mathbb E^{\mu_f}
\left[
\left.
-\log\frac{R_n}{R_{n-1}}
\right|
\mathcal F_{n-1}
\right]
\\
&\geq
-\mathbb E^{\mu_f}
\left[
\left.
-\log\frac{R_n}{R_{n-1}}
\right|
\mathcal F_{n-1}
\right]
\\
&=-(L_n-L_{n-1})
\\
&=-s\bigl(p_0(Z_{n-1})\bigr)
\\
&\geq-\log2
\end{aligned}
\]
almost surely.  Thus Lemma~\ref{lem:no-plus-infty} applies with
$B=\log2$.

When $\alpha=-1$, the endpoint value $p_0(-1)=0$ may occur.  A transition
whose conditional probability is zero is selected only on a
$\mu_f$-null event.  Hence the logarithmic identities above, and in
particular the martingale-increment estimate, remain valid
$\mu_f$-almost surely.  Notice also that the entropy causes no difficulty,
since $s(0)=s(1)=0$.

Moreover, by the definition of $\theta_1$,
\[
        L_n
        =
        \sum_{k=1}^n s\bigl(p_0(Z_{k-1})\bigr)
        \leq n\theta_1.
\]
Set
\[
        d_1:=\frac{\theta_1}{\log2}.
\]

If $d_1=1$, then
\[
        \Cap_1\bigl(\Lambda^{-1}(K)\bigr)
        \leq \Cap_1(\partial T_\ast)=0
\]
for every Borel set $K\subseteq[0,1]$, so we may take $K_0=[0,1]$.

Assume now that $0\leq d_1<1$.  Lemma~\ref{lem:no-plus-infty} gives
\[
        \mu_f\left(\bigcup_{m=1}^\infty E_m\right)=1,
        \qquad
        E_m:=\{x:M_n(x)\leq m\text{ for infinitely many }n\}.
\]
By Lemma~\ref{lem:stopping-cylinder},
\[
        \mathcal H^{d_1}(E_m)\leq e^m<\infty,
\]
and hence
\[
        \Cap_{d_1}\bigl(\Lambda^{-1}(E_m)\bigr)=0
\]
for every $m$.  Set
\[
        K_0:=\bigcup_{m=1}^\infty E_m.
\]
Then $K_0$ is Borel and $\mu_f(K_0)=1$.  By countable subadditivity of
capacity,
\[
\begin{aligned}
        \Cap_{d_1}\bigl(\Lambda^{-1}(K_0)\bigr)
        &\leq
        \sum_{m=1}^\infty
        \Cap_{d_1}\bigl(\Lambda^{-1}(E_m)\bigr)
        =0.
\end{aligned}
\]
This proves the theorem.
\end{proof}

\section{Further results and problems}\label{SectFR}
This section contains some supplementary material, mostly meant to have a map where to place the main findings of the article. Results of routine algebraic calculations are presented without lengthy derivations.
\subsection{Different coordinate systems}\label{SSectCS}
In this subsection we provide some details on the two main parameterizations of the data of the de Rham systems \eqref{eqDR}, and on their mutual relationship. This is useful to chart known results, problems which are still open, and provide a guide for generalizations and extensions.
\subsubsection{Linear coordinates}
Let
\[
        \mathcal D
        =
        \{(F_0,F_1):\text{\rm (A1) and (A2) hold}\}
\]
be the space of de Rham data, and let
\[
        \mathcal D_{\mathrm{tr}}
        =
        \{(F_0,F_1)\in\mathcal D:\text{\rm (AA3) holds}\}.
\]
We use the linear parametrization
\[
        L:\mathcal L\to\mathcal D,
        \qquad
        (\lambda,\gamma,q)\mapsto (F_0,F_1),
\]
given by \eqref{eqLinPar}, where
\[
        \mathcal L=(0,\infty)\times(1,\infty)\times(0,\infty).
\]

\begin{proposition}\label{propTrapped}
The de Rham data corresponding to \((\lambda,\gamma,q)\in\mathcal L\)
satisfy the trapping condition \(\alpha\ge-1\) if and only if
\[
        \lambda\ge1,
        \qquad
        q\ge\gamma-1.
\]
Equivalently, if
\[
        \mathcal L_{\mathrm{tr}}
        =
        \{(\lambda,\gamma,q)\in\mathcal L:
        \lambda\ge1,\ q\ge\gamma-1\},
\]
then
\[
        L(\mathcal L_{\mathrm{tr}})=\mathcal D_{\mathrm{tr}}.
\]
\end{proposition}

\begin{proof}
For the matrices in \eqref{eqLinPar}, a direct computation gives
\[
        \Phi({}^tA_0;y)=\frac{y+\gamma-\lambda}{\lambda},
        \qquad
        \Phi({}^tA_1;y)
        =
        \frac{(q-1)y+q-\gamma}{y+\gamma}.
\]
Hence the relevant fixed points are
\[
        r_0=
        \begin{cases}
        \dfrac{\gamma-\lambda}{\lambda-1},& \lambda\ne1,\\[0.6em]
        +\infty,& \lambda=1,
        \end{cases}
        \qquad
        r_1=q-\gamma,
\]
where \(r_1\) is the fixed point of \(\Phi({}^tA_1;\cdot)\) different
from \(-1\). Since \(\gamma>1\), these formulas give
\[
        r_0\ge -1\iff \lambda\ge1,
        \qquad
        r_1\ge -1\iff q\ge\gamma-1.
\]
As \(\alpha=\min\{0,r_0,r_1\}\), the proposition follows.
\end{proof}
\subsubsection{From linear coordinates to projective linear coordinates}
We now quotient the linear parameter space by the projective action
generated by the transformations \(T_t\) in \eqref{eqTa}. If
\[
        B_j=T_tA_jT_t^{-1},
        \qquad t>0,
\]
and if \(A_j\) have linear parameters \((\lambda,\gamma,q)\), then
\(B_j\) have linear parameters
\[
        (\lambda_t,\gamma_t,q_t)
        =
        \bigl(\lambda,\ 1+t(\gamma-1),\ tq\bigr).
\]
Thus \(\lambda\) is invariant along projective orbits, while
\(\gamma-1\) and \(q\) are multiplied by the same factor.

 In order to quotient out the projective parameter \(t\), we choose the
cross-section \(q=1\). On this cross-section the admissible parameters are
\[
\mathcal A=\{(\lambda,\gamma):\lambda>0,\ 1<\gamma\},
\]
where \((\lambda,\gamma)\) represents the orbit of
\((\lambda,\gamma,1)\). The trapping region on this cross-section is
\[
\mathcal A_{\mathrm{tr}}
=
\{(\lambda,\gamma):\lambda\ge1,\ 1<\gamma\le2\}.
\]
\subsubsection{Comparing linear coordinates and potential theoretic coordinates}
Not all de Rham systems can be written in potential theoretic coordinates. In fact, the map $\mathcal{P}\xrightarrow{P}\mathcal{D}$ mapping $(\lambda,\mu,t)$ to the data provided by \eqref{eq:B0} and \eqref{eq:B1} is neither surjective, nor injective.
\begin{proposition}[The capacitary subfamily inside the universal family]
\label{prop:cap-subfamily}
The capacitary/projective parametrization with $\lambda>0$, $\mu>0$ and
$t>0$ covers exactly the region
\[
\mathcal{L}_{\text{pt}}=\left\{
(\lambda,\gamma,q):
\lambda>0,\ \gamma>1,\ q>0,\ 
q\ge \frac{4(\gamma-1)}{\lambda}
\right\}
\]
of the universal parameter space $\mathcal{L}$. That is, $P(\mathcal{P})\subsetneq\mathcal{D}$.

Moreover, under the map \(L^{-1}\circ P:\mathcal P\to{\mathcal L}_{\text{pt}}\), the points where $q>\frac{4(\gamma-1)}{\lambda}$ have two preimages in $\mathcal{P}$, while those for which $q=\frac{4(\gamma-1)}{\lambda}$ have just one preimage there. The latter case corresponds to $\lambda=\mu$ in $\mathcal{P}$.
\end{proposition}
\begin{proof}
    Let \((\lambda,\mu,t)\in\mathcal P\), and let
\((\lambda,\gamma,q)\) be the corresponding linear parameters.  From
\eqref{eq:B0} and \eqref{eq:B1}, after normalizing the matrix \(B_1\) so
that its upper right entry is \(1\), one obtains
\[
        \gamma
        =
        1+t\frac{\lambda^2}{\lambda+\mu},
        \qquad
        q
        =
        t\frac{\lambda+\mu}{\mu}.
\]
Hence
\[
        \frac{q}{\gamma-1}
        =
        \frac{(\lambda+\mu)^2}{\mu\lambda^2},
\]
which implies 
\((\lambda,\gamma,q)\in\mathcal{L}_{\text{pt}}\).

Conversely, suppose
\[
        \lambda>0,\qquad \gamma>1,\qquad q>0,
        \qquad
        q\ge \frac{4(\gamma-1)}{\lambda}.
\]
Set
\[
        R=\frac{\lambda q}{\gamma-1}.
\]
Then \(R\ge4\). Choose \(x>0\) such that
\begin{equation}\label{eqSDEq}
        x+\frac1x+2=R,
\end{equation}
and put
\[
        \mu=\lambda x,
        \qquad
        t=(\gamma-1)\frac{\lambda+\mu}{\lambda^2}.
\]
Then \(\mu>0\), \(t>0\), and the above formulas give
\[
        1+t\frac{\lambda^2}{\lambda+\mu}=\gamma
\]
and
\[
        t\frac{\lambda+\mu}{\mu}
        =
        (\gamma-1)\frac{(\lambda+\mu)^2}{\mu\lambda^2}
        =
        q.
\]
Thus every point satisfying the stated inequality is obtained from
potential theoretic parameters.

The last assertion follows from a simple analysis of the equation \eqref{eqSDEq}, having as unknown $x=\frac{\mu}{\lambda}$.
\end{proof}

\subsection{Improving the results by quotienting out the projective coordinate}
\label{SSectOpt}

The conclusions of Theorem~\ref{theoMain}, as well as those of
Okamura's Theorem A, are invariant under conjugation by the maps \(T_t\).
It is therefore natural to optimize the constants \(\theta_0,\theta_1\)
along each projective orbit.

We use the projectivized linear coordinates introduced in
\S~\ref{SSectCS}. Thus
\[
        \mathcal A_{\mathrm{tr}}
        =
        \{(\lambda,\gamma):\lambda\ge1,\ 1<\gamma\le2\},
\]
and the point \((\lambda,\gamma)\in\mathcal A_{\mathrm{tr}}\) represents
the orbit of the normalized linear parameters
\[
        (\lambda,\gamma,1)\in\mathcal L_{\mathrm{tr}}.
\]
Under conjugation by \(T_t\), this point is sent to
\[
        (\lambda_t,\gamma_t,q_t)
        =
        \bigl(\lambda,\ 1+t(\gamma-1),\ t\bigr),
        \qquad t>0.
\]
Let \(\theta_0(t)\) and \(\theta_1(t)\) denote the constants in
\eqref{eqTheta} corresponding to these parameters. We define the
projectively optimized constants by
\[
        \Theta_0(\lambda,\gamma)
        =
        \sup_{t>0}\theta_0(t),
        \qquad
        \Theta_1(\lambda,\gamma)
        =
        \inf_{t>0}\theta_1(t).
\]

\begin{theorem}[Projectively optimized estimates]
\label{theoProjOpt}
Let \((\lambda,\gamma)\in\mathcal A_{\mathrm{tr}}\), and let \(f\) be
any de Rham solution whose data lie on the projective orbit represented
by \((\lambda,\gamma)\). Then the capacitary estimate
\eqref{eqSynthTwo} can be sharpened along the projective orbit to
\begin{equation}\label{eqSynthTwoBis}
\frac{\Theta_0(\lambda,\gamma)}{\log2}
\le
\inf\left\{
s\ge0:
\exists K\subseteq[0,1]\ \text{\rm Borel},\
\mu_f(K)=1,\
\Cap_s(\Lambda^{-1}(K))=0
\right\}
\le
\frac{\Theta_1(\lambda,\gamma)}{\log2}.
\end{equation}
The same replacement of \(\theta_0,\theta_1\) by
\(\Theta_0,\Theta_1\) gives the corresponding projectively optimized
version of Okamura's Hausdorff-dimensional estimate.

Moreover:
\begin{enumerate}
\item[\rm(i)]
The equality
\[
        \Theta_0(\lambda,\gamma)=\Theta_1(\lambda,\gamma)
\]
holds if and only if
\begin{equation}\label{eqSharpCurveA}
        \lambda(2-\gamma)=1.
\end{equation}
Equivalently, in \(\mathcal A_{\mathrm{tr}}\), this is the curve
\[
        1<\gamma<2,
        \qquad
        \lambda=\frac1{2-\gamma}.
\]
On this curve,
\[
        \Theta_0(\lambda,\gamma)
        =
        \Theta_1(\lambda,\gamma)
        =
        s(2-\gamma)
        =
        s\left(\frac1\lambda\right).
\]

\item[\rm(ii)]
The optimized upper constant is trivial, namely
\[
        \Theta_1(\lambda,\gamma)=\log2,
\]
if and only if
\begin{equation}\label{eqBadRegionCondition}
        \left(\frac1\lambda-\frac12\right)
        \left(2-\gamma-\frac12\right)
        \le0.
\end{equation}
\end{enumerate}
\end{theorem}
\begin{proof}
We first compute \(\Theta_0,\Theta_1\). For the parameters
\[
        (\lambda_t,\gamma_t,q_t)
        =
        \bigl(\lambda,\ 1+t(\gamma-1),\ t\bigr),
\]
the three relevant \(p\)-coordinates of Okamura's interval are
\begin{equation}\label{eqThreePar}
        p_0(r_0)=\frac1\lambda,
        \qquad
        p_0(r_1)=2-\gamma,
        \qquad
        p_0(0)=\frac{1}{1+t(\gamma-1)}.
\end{equation}
Indeed,
\[
        p_0(y)=\frac{y+1}{y+\gamma_t}.
\]
Moreover
\[
        r_0(t)=\frac{\gamma_t-\lambda}{\lambda-1}
        \quad(\lambda\neq1),
        \qquad
        r_1(t)=q_t-\gamma_t.
\]
The expressions in \eqref{eqThreePar} easily follow by direct calculation. 

The first two values correspond to the fixed points \(r_0,r_1\), while
the last one corresponds to \(y=0\). The last point is the moving point
along the projective orbit. As \(t\) ranges over \((0,\infty)\), it
ranges over the whole interval \((0,1)\).

Set
\[
        J_{\lambda,\gamma}
        =
        \operatorname{conv}
        \left\{
        \frac1\lambda,\ 2-\gamma
        \right\}.
\]
Placing the moving point outside \(J_{\lambda,\gamma}\) only enlarges the
interval on which the entropy function is minimized or maximized.
Therefore the optimal choice, both for maximizing the lower constant and
for minimizing the upper constant, is to place
\[
        \frac{1}{1+t(\gamma-1)}
\]
inside \(J_{\lambda,\gamma}\). It follows that
\begin{equation}\label{eqThetaOpt}
        \Theta_0(\lambda,\gamma)
        =
        \min_{p\in J_{\lambda,\gamma}}s(p),
        \qquad
        \Theta_1(\lambda,\gamma)
        =
        \max_{p\in J_{\lambda,\gamma}}s(p).
\end{equation}

Equivalently, if
\[
        a_{\lambda,\gamma}
        =
        \min\left\{\frac1\lambda,\ 2-\gamma\right\},
        \qquad
        b_{\lambda,\gamma}
        =
        \max\left\{\frac1\lambda,\ 2-\gamma\right\},
\]
then
\[
        J_{\lambda,\gamma}
        =
        [a_{\lambda,\gamma},b_{\lambda,\gamma}],
\]
and
\[
        \Theta_0(\lambda,\gamma)
        =
        \min\{s(a_{\lambda,\gamma}),s(b_{\lambda,\gamma})\},
\]
whereas
\[
        \Theta_1(\lambda,\gamma)
        =
        \begin{cases}
        \log2,
        &\text{if } \dfrac12\in J_{\lambda,\gamma},\\[0.6em]
        \max\{s(a_{\lambda,\gamma}),s(b_{\lambda,\gamma})\},
        &\text{if } \dfrac12\notin J_{\lambda,\gamma}.
        \end{cases}
\]

The optimizing projective parameters are explicit. Put
\[
        m(t)=\frac{1}{1+t(\gamma-1)}.
\]
Then \(m(t)\in J_{\lambda,\gamma}\) if and only if \(t\) is optimal for
both problems. More precisely, for every
\[
        m\in J_{\lambda,\gamma}\cap(0,1),
\]
the value
\begin{equation}\label{eqOptimalQ}
        t=q=
        \frac{1-m}{m(\gamma-1)}
\end{equation}
places the moving point at \(m\), and hence realizes both
\(\Theta_0(\lambda,\gamma)\) and \(\Theta_1(\lambda,\gamma)\). Thus the
optimizing value of the projective coordinate is generally not unique:
the full set of optimal values is obtained from
\eqref{eqOptimalQ} as \(m\) ranges over
\(J_{\lambda,\gamma}\cap(0,1)\).

The capacitary estimate \eqref{eqSynthTwoBis} follows immediately by
applying \eqref{eqSynthTwo} along the projective orbit and then optimizing
the lower and upper constants over \(t>0\).

We next identify the locus where the optimized lower and upper constants
coincide. Since the entropy function \(s\) is strictly concave on
\((0,1)\), symmetric with respect to \(1/2\), strictly increasing on
\((0,1/2)\), and strictly decreasing on \((1/2,1)\), the equality
\[
        \Theta_0(\lambda,\gamma)
        =
        \Theta_1(\lambda,\gamma)
\]
holds if and only if the interval \(J_{\lambda,\gamma}\) degenerates to
one point. Hence
\[
        \frac1\lambda=2-\gamma.
\]
This ends the proof of (i).

Finally, we identify the region where the optimized upper estimate is
still trivial. By \eqref{eqThetaOpt},
\[
        \Theta_1(\lambda,\gamma)=\log2
\]
if and only if
\[
        \frac12\in
        \operatorname{conv}
        \left\{
        \frac1\lambda,\ 2-\gamma
        \right\}.
\]
Equivalently,
\[
        \left(\frac1\lambda-\frac12\right)
        \left(2-\gamma-\frac12\right)
        \le0.
\]
This ends the proof of (ii).
\end{proof}

\subsection{Some open problems}
\subsubsection{Problems related to Theorem \ref{theoMain}}
The natural meta-problem is to determine, for
\((\lambda,\gamma)\in\mathcal A_{\mathrm{tr}}\), the exact value of
\[
\Gamma(\lambda,\gamma):=\inf\left\{
s\ge0:
\exists K\subseteq[0,1]\ \text{Borel},\ 
\mu_f(K)=1,\ 
\Cap_s(\Lambda^{-1}(K))=0
\right\},
\]
where \(f\) corresponds to \eqref{eqDR} with the parameters
\((\lambda,\gamma)\).
By the results of the previous section, we only know the answer when
\(\lambda(2-\gamma)=1\). The case \(\lambda=\mu\) is especially relevant
from the viewpoint of potential theory. In the normalized section
\(q=1\), it corresponds to \(\gamma=1+\lambda/4\). The only presently
understood point on this diagonal is the logarithmic case
\(\lambda=\mu=1\), for which \(\Gamma=0\).

\begin{enumerate}
    \item[(i)] Find the exact value of \(\Gamma(\lambda,\gamma)\) along
    the diagonal \(\lambda=\mu=2^\sigma\), \(0<\sigma<1\). In the
    normalized section \(q=1\), this corresponds to
    \(\gamma=1+2^{\sigma-2}\).

    \item[(ii)] In the logarithmic case \(\lambda=\mu=1\), we expect that,
    by using suitable kernels for the definition of the capacity, one
    could find some sharper information. The same problem can be posed
    regarding Theorem A, using Hausdorff measures with gauges which are
    finer than the power gauge.

    \item[(iii)] When one leaves the trapping range
    \(\mathcal A_{\mathrm{tr}}\), the left-continuous solution has jumps
    at the dyadic points. Its range is therefore obtained from \([0,1]\)
    by deleting the corresponding jump intervals, with endpoint
    conventions depending on the chosen left-continuous normalization.
    One expects a corresponding collapse in the minimal capacitary size
    of a Borel set carrying full \(\mu_f\)-mass. The problem is to make
    this expectation precise, or to disprove it.
    \item[(iv)] Okamura's lower estimate states that
\[
\dim_H(K)<\frac{\theta_0}{\log2}
\quad\Longrightarrow\quad
\mu_f(K)=0.
\]
Determine whether this estimate admits a capacitary refinement at the
lower endpoint. More precisely, is it true that, for every Borel set
\(K\subseteq[0,1]\),
\[
\Cap_{\frac{\theta_0}{\log2}}
\bigl(\Lambda^{-1}(K)\bigr)=0
\quad\Longrightarrow\quad
\mu_f(K)=0?
\]
Equivalently, does \(\mu_f\) vanish on all sets that are polar for the
critical dyadic capacity
\(\Cap_{\theta_0/\log2}\)? If this fails in general, characterize the
parameters for which it holds and determine the appropriate finer
capacity or gauge at the lower endpoint.
\end{enumerate}
Taken together, Theorems \ref{theodrCap} and \ref{theoMain} (or Theorem A)
imply that the distributional derivatives of the functions
\[
c(x):=\Cap_s(\Lambda^{-1}([0,x)))
\]
and
\[
q(x):=\Qap_{\frac{1-s}{2}}([0,x))
\]
are mutually singular for \(0\le s<1\). The function \(q\), in fact,
depends smoothly on \(x\) for \(x>0\). This is in the same spirit,
although much easier to solve, as
\cite[\S 16.3, Question 16.7]{LP2016}, where it is asked if two
different codings, dyadic vs. triadic, of the Cantor set produce
equilibrium measures which are mutually singular.
\subsubsection{Nonlinear capacitary de Rham systems}
In \cite{ARSW2014} recursive formulas for capacities on tree boundaries are proved in a rather general framework, and it is shown how, for a more restricted class, tree capacities are related to Riesz type capacities on metric spaces. The de Rham systems in Theorem \ref{theodrCap} are modeled on the dyadic decomposition of an interval, and on linear capacities, but we might think of different decompositions of spaces having larger dimension, or of nonlinear capacities. 

We make here some considerations related to nonlinear potential theories related to the dyadic decomposition of the interval.
They lead to de Rham systems whose defining maps are in general not fractional linear.

Fix $1<p<\infty$ and let $p'$ be its conjugate exponent, $1/p+1/p'=1$. For a weight $\wt$ on $E(T_\ast)$,  we define the $(p,\wt)$-capacity of a Borel set $E\subseteq\partial T_\ast$ to be
\[
\Cap_*^{\wt}(E)=\inf\left\{\sum_{{\mathrm a}\in E(T_*)} f({\mathrm a})^p\wt({\mathrm a}):\ f\ge 0,\ \sum_{{\mathrm a}\in P_*(\zeta)} f({\mathrm a})\wt({\mathrm a})\ge 1\ \text{for all }\zeta\in E\right\}.
\]
The definition of $\Cap_*^{\wt}$ indeed depends on $p$. 
A special case of Theorem 30 in \cite{ARSW2014} is
\[
    \Cap^\wt_{\mathrm{a}}(E)=\frac{\Cap^\wt_{\mathrm{a0}}(E\cap\partial T_{a0})+\Cap^\wt_{\mathrm{a1}}(E\cap\partial T_{a1})}{\left[1+\wt(a)(\Cap^\wt_{\mathrm{a0}}(E\cap\partial T_{a0})+\Cap^\wt_{\mathrm{a1}}(E\cap\partial T_{a1}))^{p'-1}\right]^{p-1}}.
\]
Let, as before, $\wt=\wt^{\lambda,\mu}=\lambda^{|a|_0}\mu^{|a|_1}$. Then we have the nonlinear version of \eqref{eqRicorso},
\[
\Cap^{\lambda,\mu}_\ast(\tau_0(E)\cup\tau_1(F))=
\frac{\lambda^{1-p}\Cap_\ast^{\lambda,\mu}(E)
+\mu^{1-p}\Cap_\ast^{\lambda,\mu}(F)}
{\left[
1+
\left(
\lambda^{1-p}\Cap_\ast^{\lambda,\mu}(E)
+\mu^{1-p}\Cap_\ast^{\lambda,\mu}(F)
\right)^{p'-1}
\right]^{p-1}}.
\]
From this it is easy to compute
\[
C:=\Cap_\ast^{\lambda,\mu}(\partial T_\ast)=\begin{cases}
    \left(1-\left(\lambda^{1-p}+\mu^{1-p}\right)^{1-p'}\right)^{p-1}\text{ if }\lambda^{1-p}+\mu^{1-p}>1,\\
    0\text{ otherwise.}
\end{cases}
\]
We shall consider here the case $C>0$, which is the interesting one from the viewpoint of potential theory. After defining 
\begin{equation}\label{eqdRnl}
f(x):=C^{-1}\Cap_\ast^{\lambda,\mu}(\Lambda^{-1}([0,x))),
\end{equation}
proceeding as in the case $p=2$ one shows that $f$ is a solution of the de Rham system \eqref{eqDR}, with
\begin{equation}\label{eqFzo}
F_0(x)=\frac{x}{[\lambda+C^{p'-1}x^{p'-1}]^{p-1}},\ F_1(x)=\frac{\lambda^{1-p}+\mu^{1-p}x}{[1+C^{p'-1}(\lambda^{1-p}+\mu^{1-p}x)^{p'-1}]^{p-1}}.
\end{equation}
The fact that $f$ is strictly increasing, left continuous, and $f(j)=j$ for $j=0,1$ follows from basic properties in potential theory. From potential theory it also follows that $f$ is continuous if $\lambda,\mu\ge1$. In fact, if $\zeta\in\partial T_\ast$, 
\[
\Cap^{(\wt,p)}_\ast(\{\zeta\})=\frac{1}{\left(\sum_{a\in P_\ast(\zeta)}\wt(a)\right)^{p-1}},
\]
which vanishes for all $\zeta$ if $\wt=\wt^{\lambda,\mu}$ with $\lambda,\mu\ge1$. In fact,
\[
\begin{split}
\Cap_\ast^{\lambda,\mu}(\Lambda^{-1}([0,x)))
&\le
\lim_{y\to x^+}
\Cap_\ast^{\lambda,\mu}(\Lambda^{-1}([0,y)))        \\
&=
\Cap_\ast^{\lambda,\mu}(\Lambda^{-1}([0,x]))        \\
&\le
\Cap_\ast^{\lambda,\mu}(\Lambda^{-1}([0,x)))
+
\Cap_\ast^{\lambda,\mu}(\Lambda^{-1}(\{x\}))        \\
&=
\Cap_\ast^{\lambda,\mu}(\Lambda^{-1}([0,x))).
\end{split}
\]
whenever $\Cap_\ast^{\lambda,\mu}(\Lambda^{-1}(\{x\}))=0$.

Also, from the de Rham equation for $f$ itself it is easy to deduce that $f$ has jumps at the dyadic points if $\lambda<1$. We summarize these considerations in a statement.
\begin{theorem}\label{theoCapNl}
  Let $1<p<\infty$.  Suppose $\lambda,\mu>0$ are such that $C>0$, and let $F_0,F_1:[0,1]\to[0,1]$ be defined by \eqref{eqFzo}. Then, the function $f:[0,1]\to[0,1]$ defined in \eqref{eqdRnl} is strictly increasing, left continuous, and it satisfies
  \[
  f(g_j(x))=F_j(f(x)),\ f(j)=j,
  \]
  for $j=0,1$. If $\lambda,\mu\ge1$, then $f$ is continuous. If $0<\lambda<1$, and $x\in[0,1)$ is a dyadic point, then $f$ has jump discontinuities at $x$.
\end{theorem}
This is what potential theory tells us without effort about the function $f$ in the nonlinear case. The problem is to study the regularity of \(f\).

\end{document}